\documentclass[11pt]{article}
\usepackage{amsfonts}
\usepackage{amsmath}

\newcommand{\be}{\begin{equation}}
\newcommand{\ee}{\end{equation}}
\newcommand{\bea}{\begin{eqnarray}}
\newcommand{\eea}{\end{eqnarray}}
\newcommand{\beas}{\begin{eqnarray*}}
\newcommand{\eeas}{\end{eqnarray*}}

\newtheorem{theorem}{Theorem}[section]
\newtheorem{definition}[theorem]{Definition}
\newtheorem{proposition}[theorem]{Proposition}

\newtheorem{remark}[theorem]{Remark}
\newtheorem{example}[theorem]{Example}
\newtheorem{examples}[theorem]{Examples}
\newtheorem{foo}[theorem]{Remarks}

\newenvironment{proof}{\addvspace{\medskipamount}\par\noindent{\it Proof}.}
{\unskip\nobreak\hfill$\Box$\par\addvspace{\medskipamount}}

\newcommand{\abs}[1]{\left|#1\right|}     

\DeclareMathOperator{\esssup}{ess\,sup}

\begin{document}
\title{Existence of solution to scalar BSDEs \\with weakly $L^{1+}$-integrable terminal values}

\author{Ying Hu\thanks{IRMAR,
Universit\'e Rennes 1, Campus de Beaulieu, 35042 Rennes Cedex, France (ying.hu@univ-rennes1.fr) 
and School of
Mathematical Sciences, Fudan University, Shanghai 200433, China.
Partially supported by Lebesgue center of mathematics ``Investissements d'avenir"
program - ANR-11-LABX-0020-01,  by ANR CAESARS - ANR-15-CE05-0024 and by ANR MFG - ANR-16-CE40-0015-01.} \and
Shanjian Tang\thanks{Department of Finance and Control Sciences, School of
Mathematical Sciences, Fudan University, Shanghai 200433, China (e-mail: sjtang@fudan.edu.cn).
Partially supported by National Science Foundation of China (Grant No. 11631004)
and Science and Technology Commission of Shanghai Municipality (Grant No. 14XD1400400).   }}

\maketitle

{\bf Abstract.} In this paper, we study a scalar linearly growing BSDE with a weakly $L^{1+}$-integrable terminal value. We prove that the BSDE admits a solution
if the terminal value satisfies some $\Psi$-integrability condition, which is weaker than the usual $L^p$ ($p>1$) integrability and stronger than $L\log L$ integrability.  We show by a counterexample that $L\log L$ integrability is not sufficient for the existence of solution
to a BSDE of a linearly growing generator.

{\bf AMS Subject Classification}: 60H10

{\bf Key Words} Backward stochastic differential equation, weak integrability,
terminal condition, dual representation.
\section{Introduction}

Consider the following Backward Stochastic Differential Equation (BSDE):
\begin{equation}\label{EDSR}
Y_t=\xi+\int_t^T f(s,Y_s,Z_s)ds-\int_t^T Z_sdW_s,
\end{equation}
where the function $f:[0,T] \times \Omega \times \mathbb{R} \times \mathbb{R}^{1\times d} \rightarrow \mathbb{R}$ satisfies
$$|f(s,y,z)|\le \alpha_s +\beta |y|+\gamma |z|,$$
with $\alpha\in L^1(0,T),\beta\ge 0$ and $\gamma>0$.

It is well known that if $\xi\in L^p$, with $p>1$ , then there exists a solution to BSDE (\ref{EDSR}), see e.g.
\cite{PP90,EPQ97, BDHPS03}.
The aim of this paper is to find a weaker integrability condition for the terminal value $\xi$, under which the solution still exists .

Set
$$\Psi_\lambda(x)=xe^{(\frac{2}{\lambda}\log(x+1))^{1/2}}, \quad (\lambda,x)\in (0,\infty)\times [0,、\infty).$$

Our sufficient condition is: there exists $\lambda\in (0,\frac{1}{\gamma^2{T}})$ such that
$$\mathbb E[\Psi_\lambda(|\xi|)]<+\infty.$$

\begin{remark}
Note that the preceding $\Psi_\lambda$-integrability is stronger than $L^1$, 
weaker than $L^p$ for any $p>1$, because for any $\varepsilon>0$, we have,
$$x\le xe^{(\frac{2}{\lambda}\log(x+1))^{1/2}}\le xe^{ {\varepsilon}\log(x+1)+\frac{1}{2\varepsilon\lambda}}
\le e^{\frac{1}{2\varepsilon\lambda}} x(x+1)^\varepsilon,\quad x\ge 0.$$
Moreover, for any $p\ge 1$, there exists a constant $C_p>0$ such that
$$xe^{(\frac{2}{\lambda}\log(x+1))^{1/2}}\ge C_p x\log^p(x+1).$$
We will see  that even the condition $$\mathbb E[\Psi_\lambda(|\xi|)]<+\infty， $$ for a certain
$\lambda>\frac{1}{\gamma^2{T}}$ (which implies that $|\xi|\log^p(|\xi|+1)\in L^1$) is still too weak to ensure the existence of solution by giving a simple example in Example \ref{example}.
\end{remark}

Note that if the generator $f$ is of sublinear growth with respect to $z$, i.e. there exists $q\in [0,1)$,
$$|f(t,y,z)|\le \alpha+\beta |y|+\gamma |z|^q,$$
then there exists a solution for $\xi\in L^1$, see \cite{BDHPS03}.

Our method applies the dual representation of solution to BSDE with convex generator (see, e.g. \cite{EPQ97, T06, DHR11}) in order to establish some a priori estimate and then the localization procedure
of real-valued BSDE \cite{BH06}.

Let us close this introduction by giving the notations that we will use in all the paper. For the remaining of the paper, let us fix a nonnegative real number $T>0$. First of all, $(W_t)_{t \in [0,T]}$ is a standard Brownian motion with values in $\mathbb{R}^d$ defined on some complete probability space $(\Omega,\mathcal{F},\mathbb{P})$. $(\mathcal{F}_t)_{t \ge 0}$ is the natural filtration of the Brownian motion $W$ augmented by the $\mathbb{P}$-null sets of $\mathcal{F}$. The sigma-field of predictable subsets of $[0,T] \times \Omega$ is denoted by $\mathcal{P}$.

Consider real valued BSDEs which are equations of type (\ref{EDSR}), where $f$ (hereafter called the generator) is a random function $[0,T] \times \Omega \times \mathbb{R} \times \mathbb{R}^{1\times d} \rightarrow \mathbb{R}$ and measurable with respect to $\mathcal{P} \otimes \mathcal{B}(\mathbb{R}) \otimes \mathcal{B}(\mathbb{R}^{1\times d})$, and $\xi$ (hereafter called the terminal condition or terminal value) is a real $\mathcal{F}_T$-measurable random variable.

\begin{definition}
By a solution to  BSDE (\ref{EDSR}), we mean a pair $(Y_t,Z_t)_{t \in [0,T]}$ of predictable processes with values in $\mathbb{R} \times \mathbb{R}^{1\times d}$ such that $\mathbb{P}$-a.s., $t \mapsto Y_t$ is continuous, $t \mapsto Z_t$ belongs to $L^2(0,T)$ and $t \mapsto g(t,Y_t,Z_t)$ belongs to $L^1(0,T)$, and $\mathbb{P}$-a.s. $(Y,Z)$ verifies  (\ref{EDSR}). 
\end{definition}

By BSDE ($\xi$,$f$), we mean the BSDE of generator $f$ and terminal condition $\xi$.

For any real $p\ge 1$, $\mathcal{S}^p$ denotes the set of real-valued, adapted and c\`adl\`ag processes $(Y_t)_{t \in [0,T]}$ such that
$$||Y||_{\mathcal{S}^p}:=\mathbb{E} \left[\sup_{0 \leq t\leq T} \abs{Y_t}^p \right]^{1/p} < + \infty,$$
and ${\cal M}^p$ denotes the set of (equivalent class of) predictable processes $(Z_t)_{t \in [0,T]}$ with values in $\mathbb{R}^{1 \times d}$ such that
$$||Z||_{{\cal M}^p}:=\mathbb{E}\left[\left(\int_0^T \abs{Z_s}^2 ds \right)^{p/2}\right]^{1/p} < +\infty.$$

The rest of the paper is organized as follows. Section 2 establishes a necessary and sufficient condition for the
existence of solution to BSDE (\ref{EDSR}) for the typical form of generator $f(t,y,z)=\alpha_t+\beta y+\gamma |z|$. Section 3 gives the $\Psi_\lambda$ integrability condition for the
existence of solution to BSDE (\ref{EDSR}) for  $f(t,y,z)=\alpha_t+\beta y+\gamma |z|$. Section 4 is devoted to the  sufficiency of the $\Psi_\lambda$-integrability condition for 
the existence of solution to BSDE (\ref{EDSR}) of the general linearly growing generator. 

\section{Typical Case}
Let us first consider the following BSDE:
\begin{equation}\label{bsdez}
Y_t=\xi+\int_t^T (\alpha_s+\beta Y_s+\gamma |Z_s|)ds-\int_t^T Z_sdW_s,
\end{equation}
where $\alpha\in L^1(0,T)$, and $\beta\ge 0$ and $\gamma>0$ are some real constants.
We suppose further that the terminal condition $\xi$ is nonnegative. Note that if $Y$ is a solution belonging to class D, then 
as $e^{\beta t}Y_t$ is a local supermartingale, it is a supermatingale, from which we deduce that
$Y\ge 0$. In this subsection, we restrict ourselves to nonnegative solution.

For $\xi\in L^p$ ($p>1$),  BSDE (\ref{bsdez}) has a unique solution. It has a dual representation as follows (see, e.g. \cite{EPQ97,DHR11})
\begin{equation}\label{dual}
Y_t=\mathop{\esssup}_{q\in {\cal A}} \{ \mathbb E_q[e^{\beta(T-t)}\xi|{\cal F}_t]\}+\int_t^Te^{\beta(s-t)}\alpha_sds,
\end{equation}
where ${\cal A}$ is the set of progressively measurable processes $q$ such that $|q|\le \gamma$,
$$\frac{d\mathbb Q^q}{d\mathbb P}=M^q_T,$$
with 
$$M^q_t=\exp\{\int_0^t q_sdW_s-\frac{1}{2}\int_0^t|q_s|^2ds\}, \quad t\in [0,T],$$
and $\mathbb E_q$ is the expectation with respect to $\mathbb Q^q$.

\begin{theorem}\label{CNS}
Let us suppose that $\xi\ge 0$. Then BSDE (\ref{bsdez}) admits a solution $(Y,Z)$ such that $Y\ge 0$ if and only if there exists a locally bounded process $\bar{Y}$
such that 
$$\mathop{\esssup}_{q\in {\cal A}} \{ \mathbb E_q[e^{\beta(T-t)}\xi|{\cal F}_t]\}+\int_t^Te^{\beta(s-t)}\alpha_sds\le \bar{Y}_t.$$
\end{theorem}

\begin{proof} If BSDE (\ref{bsdez}) admits a solution $(Y,Z)$ such that $Y\ge 0$, then we define a sequence of stopping times
$$\sigma_n=T\wedge\inf\{t\ge 0: |Y_t|>n\},$$
with the convention that $\inf\emptyset=+\infty$.

As $W_s^q=W_s-\int_0^s q_rdr$ is a Brownian motion under $Q^q$, 
we have
$$
Y_{t\wedge\sigma_n}=Y_{\sigma_n}+\int_{t\wedge\sigma_n}^{\sigma_n} (\alpha_s+\beta Y_s+\gamma |Z_s|-Z_sq_s)ds-\int_{t\wedge\sigma_n}^{\sigma_n}Z_sdW_s^q.
   $$
Applying It\^o's formula to $e^{\beta s}Y_s$, we deduce 
$$e^{\beta (t\wedge\sigma_n)}Y_{t\wedge\sigma_n}=e^{\beta\sigma_n}Y_{\sigma_n}+\int_{t\wedge\sigma_n}^{\sigma_n}
e^{\beta s}(\alpha_s+\gamma |Z_s|-Z_sq_s)ds-\int_{t\wedge\sigma_n}^{\sigma_n}e^{\beta s}Z_sdW_s^q,$$
from which we obtain
$$ \mathbb E_q[e^{\beta(\sigma_n-t\wedge\sigma_n)}Y_{\sigma_n}+\int_{t\wedge\sigma_n}^{\sigma_n}e^{\beta(s-t\wedge\sigma_n)}\alpha_sds|{\cal F}_t]\le Y_{t\wedge\sigma_n}.$$
Fatou's lemma yields that 
$$ \mathbb E_q[e^{\beta(T-t)}\xi|{\cal F}_t]+\int_t^Te^{\beta(s-t)}\alpha_sds\le Y_t.$$

On the other hand, if 
there exists a locally bounded process $\bar{Y}$
such that 
$$\mathop{\esssup}_{q\in {\cal A}} \{ \mathbb E_q[e^{\beta(T-t)}\xi|{\cal F}_t]\}+\int_t^Te^{\beta(s-t)}\alpha_sds\le \bar{Y}_t,$$
then we construct the solution by use of a localization method (see e.g. \cite{BH06}). We describe this method here for completeness.
 Let 
$(Y^n,Z^n)$ be the unique solution in ${\cal S}^2\times {\cal M}^2$ of the following BSDE
$$Y^n_t=\xi{\bf 1}_{\{\xi\le n\}}+\int_t^T (\alpha_s+\beta Y^n_s+\gamma|Z^n_s|)ds-\int_t^T Z^n_sdW_s.$$
By comparison theorem, $Y^n$ is nondecreasing with respect to $n$. Moreover, 
setting $q_s^n=\gamma     \ {\mbox sgn}(Z_s^n)$, we obtain
\begin{eqnarray*}
Y^n_t&=& \mathbb E_{q^n}[e^{\beta(T-t)}\xi{\bf 1}_{\{\xi\le n\}}|{\cal F}_t]+\int_t^Te^{\beta(s-t)}\alpha_sds\\
&\le & \mathbb E_{q^n}[e^{\beta(T-t)}\xi|{\cal F}_t]\}+\int_t^Te^{\beta(s-t)}\alpha_sds\\
&\le& \bar{Y}_t.
\end{eqnarray*}

Set
$$\tau_k=T\wedge \inf\{t\ge 0: \bar{Y}_t>k\},$$
and 
$$Y_k^n(t)=Y_{t\wedge \tau_k}^n,\quad Z_k^n(t)=Z_t^n{\bf 1}_{t\le\tau_k}.$$
Then $(Y_k^n,Z_k^n)$ satisfies
\begin{equation}\label{bsdekn}
Y_k^n(t)=Y_k^n(T)+\int_t^T{\bf 1}_{s\le\tau_k} (\alpha_s+\beta Y_k^n(s)+\gamma |Z_k^n(s)|)ds-\int_t^T Z_k^n(s)dW_s.
\end{equation}

For fixed $k$, $Y_k^n$ is nondecreasing with respect to $n$ and remains bounded by $k$.
We can now apply the stability property of BSDE with bounded terminal data (see e.g.
 Lemma 3, page 611 in \cite{BH06}).
Setting $Y_k(t)=\sup_n Y_k^n(t)$, there exists $Z_k$ such that $\lim_n Z_k^n=Z_k$ in ${\cal M}^2$ and
\begin{equation}\label{bsdek}
Y_k(t)=\sup_n Y_{\tau_k}^n+\int_t^{\tau_k} (\alpha_s+\beta Y_k(s)+\gamma |Z_k(s)|)ds-\int_t^{\tau_k} Z_k(s)dW_s.
\end{equation}

Finally, noting that
$$Y_{k+1}(t\wedge\tau_k)=Y_k(t\wedge\tau_k),\quad Z_{k+1}{\bf 1}_{t\le\tau_k}=Z_k{\bf 1}_{t\le\tau_k},$$
we conclude the existence of solution $(Y,Z)$.

\end{proof}

\begin{remark} Consider the case $d=1$.
If BSDE (\ref{bsdez}) admits a solution $(Y,Z)$ such that $Y\ge 0$, by taking $q=\gamma$ and $q=-\gamma$, we 
deduce that both $\xi e^{\gamma W_T}$ and $\xi e^{-\gamma W_T}$ are in $L^1(\Omega)$, which implies that $\xi e^{\gamma |W_T|}\in L^1(\Omega)$, as
$$\xi e^{\gamma |W_T|}\le \xi e^{\gamma W_T}+\xi e^{-\gamma W_T}.$$
\end{remark}

\begin{example}\label{example}
Let us set $d=1$, $T=1$, $\beta=0$, $\gamma=1$,  $\mu\in (0,1)$, and
$$\xi=e^{\frac{1}{2}W_1^2-\mu |W_1|+\frac{1}{2}\mu^2}-1.$$ 
In this case, BSDE (\ref{bsdez}) does not admit a solution $(Y,Z)$ such that $Y\ge 0$,
as $\xi e^{|W_1|}$ does not belong to $L^1(\Omega)$ by the following direct calculus:
$$\mathbb E[\xi e^{|W_1|}]=\frac{1}{\sqrt{2\pi}}\int_{-\infty}^{+\infty}(e^{\frac{1}{2}|x|^2-\mu |x|+\frac{1}{2}\mu^2}-1)e^{|x|}e^{-\frac{1}{2}|x|^2}dx=+\infty.$$
Whereas it is straightforward to see that
for any $p\ge 1$,
$\xi\log^p(\xi+1) \in L^1(\Omega)$.

For $\lambda>1$, consider the following terminal condition 
$$\xi=e^{\frac{1}{2}W_1^2-\mu |W_1|+\frac{1}{2}\mu^2}-1, \quad \mu\in ({\frac{1}{\sqrt{\lambda}}},1).$$
We have $\Psi_\lambda(\xi)\in L^1$ by the following straightforward  calculus:
$$\mathbb E[\Psi_\lambda(\xi)]=\frac{1}{\sqrt{2\pi}}\int_{-\infty}^{+\infty}(e^{\frac{1}{2}|x|^2-\mu |x|+\frac{1}{2}\mu^2}-1)e^{{\frac{1}{\sqrt{\lambda}}}\big||x|-\mu\big|}e^{-\frac{1}{2}|x|^2}dx<+\infty                   ,$$
while BSDE (\ref{bsdez})
has no solution, in view of the fact that $\xi e^{|W_1|}$ does not belong to $L^1(\Omega)$.
\end{example}

\section{Sufficient Condition}
Let us now look for a sufficient condition for the existence of a locally bounded process $\bar{Y}$
such that $$\mathop{\esssup}_{q\in {\cal A}} \{ \mathbb E_q[e^{\beta(T-t)}\xi|{\cal F}_t]\}+\int_t^Te^{\beta(s-t)}\alpha_sds\le \bar{Y}_t.$$

For $\lambda>0$, define the functions $\Phi_\lambda$ and $\Psi_\lambda$:
$$\Phi_\lambda(x)=e^{\frac{1}{2}\lambda \log^2(x)},\quad x>0,$$ 
$$\Psi_\lambda(y)=ye^{(\frac{2}{\lambda}\log(y+1))^{1/2}}, \quad y\ge 0.$$
Then we have

\begin{proposition} For any $x\in \mathbb R$ and $y\ge 0$, we have
$$e^x y\le \Phi_\lambda(e^x)+e^{\frac{2}{\lambda}}\Psi_\lambda(y).$$
\end{proposition}

\begin{proof}
Set 
$$z=(\frac{2}{\lambda}\log(y+1))^{1/2}\ge 0,$$
then
$$ y=e^{\frac{\lambda}{2}z^2}-1.$$
It is sufficient to prove that for any $x\in \mathbb R$ and $z\ge 0$,
$$e^{\frac{1}{2}\lambda x^2-x}+(e^{\frac{\lambda}{2}z^2}-1)(e^{z+\frac{2}{\lambda}-x}-1)\ge 0.$$
It is evident to see that the above inequality holds when $z+\frac{2}{\lambda}-x\ge 0$.

Consider the case $z+\frac{2}{\lambda}-x< 0$. Then $x>z+\frac{2}{\lambda}>0$.
Hence
\begin{eqnarray*}
& &e^{\frac{1}{2}\lambda x^2-x}+(e^{\frac{\lambda}{2}z^2}-1)(e^{z+\frac{2}{\lambda}-x}-1)\\
&=&e^{\frac{1}{2}\lambda (x-\frac{1}{\lambda})^2-\frac{1}{2\lambda}}+e^{\frac{\lambda}{2}z^2+z+\frac{2}{\lambda}-x}
+1-e^{z+\frac{2}{\lambda}-x}-e^{\frac{\lambda}{2}z^2}\\
&\ge& e^{\frac{1}{2}\lambda (z+\frac{1}{\lambda})^2-\frac{1}{2\lambda}}-e^{\frac{\lambda}{2}z^2}\\
&\ge & 0.
\end{eqnarray*}
\end{proof}

\begin{proposition} Let $0<\lambda<\frac{1}{\gamma^2{T}}$. For any $q\in {\cal A}$, 
$$\mathbb E[\Phi_\lambda(e^{\int_t^T q_sdW_s})|{\cal F}_t]\le \frac{1}{\sqrt{1-\lambda\gamma^2(T-t)}}.$$
\end{proposition}
\begin{proof}
Firstly, by use of Girsanov's lemma, for $\theta\in \mathbb R$,
\begin{eqnarray*}
& &\mathbb E[e^{\theta \int_t^T q_sdW_s}|{\cal F}_t]\\
&=&\mathbb E[e^{\theta \int_t^T q_sdW_s-\frac{\theta^2}{2}\int_t^T|q_s|^2ds}e^{\frac{\theta^2}{2}\int_t^T|q_s|^2ds}|{\cal F}_t]\\
&\le&e^{\frac{\theta^2\gamma^2}{2}(T-t)}.
\end{eqnarray*}

Then we apply
$$e^{\lambda\frac{x^2}{2}}=\frac{1}{\sqrt{2\pi}}\int_{-\infty}^{+\infty} e^{\sqrt{\lambda} yx-\frac{y^2}{2}}dy$$
to deduce that
\begin{eqnarray*}
& &\mathbb E\big[\Phi_\lambda(e^{\int_t^T q_sdW_s})\big|{\cal F}_t\big]\\
&=&\mathbb E\big[e^{\frac{\lambda}{2}(\int_t^T q_sdW_s)^2}\big|{\cal F}_t\big]\\
&=& \frac{1}{\sqrt{2\pi}}\int_{-\infty}^{+\infty} \mathbb E [e^{\sqrt{\lambda} y\int_t^T q_sdW_s-\frac{y^2}{2}}|{\cal F}_t]dy\\
&\le& \frac{1}{\sqrt{2\pi}}\int_{-\infty}^{+\infty}  e^{\frac{(\sqrt{\lambda} y\gamma)^2}{2}(T-t)-\frac{y^2}{2}}dy\\
&=&\frac{1}{\sqrt{1-\lambda\gamma^2(T-t)}}.
\end{eqnarray*}
\end{proof}

Applying the above two propositions, we deduce the following sufficient condition.

\begin{theorem}\label{simple} Let us suppose that there exists $\lambda\in (0,\frac{1}{\gamma^2{T}})$ such that
$$\mathbb E[\Psi_\lambda(\xi)]<+\infty.$$
Then 
\begin{equation}\label{estimateY}
\mathop{\esssup}_{q\in {\cal A}} \{ \mathbb E_q[e^{\beta(T-t)}\xi|{\cal F}_t]\}+\int_t^Te^{\beta(s-t)}\alpha_sds\le \bar{Y}_t,
\end{equation}
with
$$\bar{Y}_t=e^{\beta(T-t)}\left(\frac{1}{\sqrt{1-\lambda\gamma^2(T-t)}}+e^{\frac{2}{\lambda}}\mathbb E[\Psi_\lambda(\xi)|{\cal F}_t]\right)+\int_t^Te^{\beta(s-t)}\alpha_sds,$$
and (\ref{bsdez}) admits a solution $(Y,Z)$ such that
$$Y_t\le {\bar Y}_t.$$
\end{theorem}
\begin{proof}
Applying the above two propositions, we deduce
\begin{eqnarray*}
\mathbb E_q[\xi|{\cal F}_t]&=&\mathbb E[M^q_T(M^q_t)^{-1}\xi|{\cal F}_t]\le \mathbb E\left[e^{\int_t^T q_sdW_s}\xi\Big| {\cal F}_t\right]
\\
&\le& \mathbb E[\Phi_\lambda( e^{\int_t^T q_sdW_s})|{\cal F}_t]+e^{\frac{2}{\lambda}}\mathbb E[\Psi_\lambda(\xi)|{\cal F}_t]\\
&\le &\frac{1}{\sqrt{1-\lambda\gamma^2(T-t)}}+e^{\frac{2}{\lambda}}\mathbb E[\Psi_\lambda(\xi)|{\cal F}_t].
\end{eqnarray*}
 Then we get (\ref{estimateY}) and the rest follows from Theorem 
\ref{CNS}.
\end{proof}

\section{General Case}
Consider the following BSDE:
\begin{equation}\label{bsdez2}
Y_t=\xi+\int_t^T f(s,Y_s,Z_s)ds-\int_t^T Z_sdW_s,
\end{equation}
where $f$ satisfies
\begin{equation}\label{lineargrowth}|f(s,y,z)|\le \alpha_s +\beta |y|+\gamma |z|,
\end{equation}
with $\alpha\in L^1(0,T),\beta\ge 0$ and $\gamma>0$.

\begin{theorem} Let $f$ be a generator which is continuous with respect to $(y,z)$ and verifies (\ref{lineargrowth}), and
$\xi$ be a terminal condition.
Let us suppose that there exists $\lambda\in (0,\frac{1}{\gamma^2{T}})$ such that
$$\mathbb E[\Psi_\lambda(|\xi|)]<+\infty.$$
Then BSDE (\ref{bsdez2}) admits a solution $(Y,Z)$ such that
$$|Y_t|\le e^{\beta(T-t)}\left(\frac{1}{\sqrt{1-\lambda\gamma^2(T-t)}}+e^{\frac{2}{\lambda}}\mathbb E[\Psi_\lambda(|\xi|)|{\cal F}_t]\right)+\int_t^Te^{\beta(s-t)}\alpha_sds.  $$
\end{theorem}

\begin{proof}
Let us fix $n\in \mathbb N^*$ and $p\in \mathbb N^*$ and set $\xi^{n,p}=\xi^+\wedge n-\xi^-\wedge p$.
Let $(Y^{n,p},Z^{n,p})$ be the unique solution in ${\cal S}^2\times {\cal M}^2$ of the BSDE $(|\xi^{n,p}|,f)$.
Set
$$\bar{f}(s,y,z)= \alpha_s +\beta y+\gamma |z|,$$
and
$(\bar{Y}^{n,p},\bar{Z}^{n,p})$ be the unique solution in ${\cal S}^2\times {\cal M}^2$ of the BSDE $(|\xi^{n,p}|,\bar{f})$.

By comparison theorem, 
$$|Y^{n,p}_t|\le |{\bar Y}^{n,p}_t|.$$
Setting $q_s^{n,p}=\gamma \ {\mbox sgn}(Z_s^{n,p})$, we obtain,
\begin{eqnarray*}
|Y^{n,p}_t|&\le& |{\bar Y}^{n,p}_t|\\
&=& \mathbb E_{q^{n,p}}\left[e^{\beta(T-t)}|\xi^{n,p}|\Big|{\cal F}_t\right]+\int_t^Te^{\beta(s-t)}\alpha_sds.\\
\end{eqnarray*}
From inequality (\ref{estimateY}),
$$|Y^{n,p}_t|\le \bar{Y}_t,$$
with
$$\bar{Y}_t=e^{\beta(T-t)}\left(\frac{1}{\sqrt{1-\lambda\gamma^2(T-t)}}+e^{\frac{2}{\lambda}}\mathbb E\left[\Psi_\lambda(|\xi|)\Big|{\cal F}_t\right]\right)+\int_t^Te^{\beta(s-t)}\alpha_sds.$$
Moreover, $Y^{n,p}$ is nondecreasing with respect to $n$, and nonincreasing with respect to $p$. 
Once again, we apply the localization method as follows to conclude the existence of solution.

Set
$$\tau_k=T\wedge \inf\{t\ge 0: \bar{Y}_t>k\},$$
and 
$$Y_k^{n,p}(t)=Y_{t\wedge \tau_k}^{n,p},\quad Z_k^{n,p}(t)=Z_t^{n,p}{\bf 1}_{t\le\tau_k}.$$
Then $(Y_k^{n,p},Z_k^{n,p})$ satisfies
\begin{equation}\label{bsdeknp}
Y_k^{n,p}(t)=Y_k^{n,p}(T)+\int_t^T{\bf 1}_{s\le\tau_k} f(s,Y_k^{n,p}(s),Z_k^{n,p}(s))ds-\int_t^T Z_k^{n,p}(s)dW_s.
\end{equation}

For fixed $k$, $Y_k^{n,p}$ is nondecreasing with respect to $n$ and nonincreasing with respect to $p$, and remains bounded by $k$.
We can now apply the stability property of BSDEs with bounded terminal data.
Setting $Y_k(t)=\inf_p\sup_n Y_k^{n,p}$, there exists $Z_k$ in ${\cal M}^2$ such that $\lim_p\lim_n Z_k^{n,p}=Z_k$ in ${\cal M}^2$ and
\begin{equation}\label{bsdek2}
Y_k(t)=\inf_p\sup_n Y_{\tau_k}^{n,p}+\int_t^{\tau_k}f(s, Y_k(s),Z_k(s))ds-\int_t^{\tau_k} Z_k(s)dW_s.
\end{equation}

Finally, noting that
$$Y_{k+1}(t\wedge\tau_k)=Y_k(t\wedge\tau_k),\quad Z_{k+1}{\bf 1}_{t\le\tau_k}=Z_k{\bf 1}_{t\le\tau_k},$$
we conclude the existence of solution $(Y,Z)$.

\end{proof}

\end{document}